\documentclass[12pt]{amsart}

\usepackage{amssymb, amscd, txfonts}
\usepackage{graphicx}
\usepackage[all]{xy}

\numberwithin{equation}{section}

\newtheorem{theorem}{Theorem}[section]
\newtheorem{proposition}[theorem]{Proposition}
\newtheorem{lemma}[theorem]{Lemma}
\newtheorem{corollary}[theorem]{Corollary}

\theoremstyle{definition}

\theoremstyle{remark}
\newtheorem{remark}[theorem]{Remark}
\newtheorem{claim}[theorem]{Claim}

\renewcommand{\hom}{\operatorname{Hom}}

\newcommand{\Z}{\mathbb{Z}}
\newcommand{\Q}{\mathbb{Q}}

\newcommand{\C}{\mathbb{C}}

\newcommand{\QZ}{\mathbb{Q}/\mathbb{Z}}
\newcommand{\Fp}{\mathbb{F}_{p}}
\newcommand{\Fpu}{\mathbb{F}_{p}^{\times}}
\newcommand{\Zp}{\mathbb{Z}_{p}}
\newcommand{\Zpu}{\mathbb{Z}_{p}^{\times}}
\newcommand{\OA}{{\rm O}(A)}
\newcommand{\OB}{{\rm O}(B)}
\newcommand{\GAOA}{G(A, {\rm O}(A))}
\newcommand{\gaoa}{G'(A, {\rm O}(A))}
\newcommand{\SL}{{\rm SL}_2(\mathbb{Z})}
\newcommand{\Mp}{{\rm Mp}_2(\mathbb{Z})}
\newcommand{\e}{{\mathbf e}}

\DeclareMathOperator{\tr}{tr}

\begin{document}

\title[]{Equivariant Gauss sum of finite quadratic forms}
\author[]{Shouhei Ma}
\thanks{Supported by Grant-in-Aid for Scientific Research (S) 15H05738.} 
\address{Department~of~Mathematics, Tokyo~Institute~of~Technology, Tokyo 152-8551, Japan}
\email{ma@math.titech.ac.jp}

\begin{abstract}
The classical quadratic Gauss sum can be thought of as an exponential sum attached to a quadratic form on a cyclic group. 
We introduce an equivariant version of Gauss sum for arbitrary finite quadratic forms, 
which is an exponential sum twisted by the action of the orthogonal group. 
We prove that simple arithmetic formulae hold for some basic classes of quadratic forms. 
In application, such invariant appears in the dimension formula for certain vector-valued modular forms. 
\end{abstract} 

\maketitle


\section{Introduction}\label{sec:intro}

In his study of the quadratic reciprocity law, Gauss introduced and evaluated the exponential sum 
$\sum_{x=0}^{p-1}e(ax^2/p)$ where $a\in{\Fpu}$, which is now called the Gauss sum. 
Here $e(z)={\rm exp}(2\pi iz)$ for $z\in{\QZ}$. 
If we consider the ${\QZ}$-valued quadratic form $q(x)=2^{-1}ax^2/p$ on the cyclic group $A={\Z}/p$ and 
the associated bilinear form $(x, y)=q(x+y)-q(x)-q(y)$, 
the Gauss sum can be written in the form $G(A)=\sum_{x\in A}e((x, x))$. 
Gauss evaluated such an exponential sum also for $A={\Z}/p^k$ (see \cite{B-E-W}). 
If we consider $G(A)$ for general finite quadratic forms $A$, we have essentially no further problem concerning evaluation 
because the product formula $G(A_1\oplus A_2)=G(A_1)\cdot G(A_2)$ holds and 
$A$ can be decomposed into cyclic forms as above and certain special forms on ${\Z}/2^k\oplus{\Z}/2^k$ 
(see \cite{Wa}). 

In this paper we introduce an equivariant version of $G(A)$ for a finite quadratic form $A=(A, q)$, 
which is a twist of $G(A)$ by the orthogonal group ${\OA}={\rm O}(A, q)$. 
Let $( \: , \: ): A\times A\to{\QZ}$ be the associated bilinear form as above. 
In general, for a subgroup $\Gamma$ of ${\OA}$, we define 
\begin{equation*}
G(A, \Gamma) = \sum_{[x]\in\Gamma\backslash A}\sum_{y\in \Gamma x} e((x, y)). 
\end{equation*}
The classical Gauss sum $G(A)$ is the case $\Gamma=\{ {\rm id} \}$. 
We are interested in the case $\Gamma={\OA}$, the motivation coming from certain modular forms (\S \ref{sec:dim formula}). 
If $A=\oplus_{p}A_p$ is the decomposition into $p$-parts, we have (Corollary \ref{localize Gauss sum})
\begin{equation*}
G(A, {\OA}) = \prod_{p}G(A_p, {\rm O}(A_p)). 
\end{equation*}
Hence we may restrict our attention to quadratic forms on $p$-groups. 

Our main result is an arithmetic formula of $G(A, {\OA})$ for some basic quadratic forms, 
which has similar but different shape from that of $G(A)$. 
For simplicity we state the result only for $p>2$, referring to \S \ref{sec:p=2} for the case $p=2$. 
We write $(\frac{\cdot}{p})$ for the Legendre symbol. 

\begin{theorem}[\S \ref{sec:p odd}]\label{main GAOA}
Let $p>2$. 

$(1)$ Let $A$ be the symmetric bilinear form $(x, y)=axy/p^k$ on ${\Z}/p^k$ where $a\in{\Zpu}$. 
Then 
\begin{equation*}
{\GAOA} = 
\begin{cases}
    0,           & p^k \equiv 3 \: \textrm{mod}\:  4,  \\ 
     \left( \frac{a}{p} \right)^k \sqrt{|A|}, & p^k \equiv 1 \: \textrm{mod}\:  4. 
\end{cases}
\end{equation*}

$(2)$ Let $A$ be a $p$-elementary form (quadratic space over ${\Fp}$) of dimension $m>1$. 
When $A$ contains an isotropic vector, then 
\begin{equation*}
{\GAOA}= 
\begin{cases}
    0,           & p \equiv 3 \: \textrm{mod} \: 4, \\ 
    2\left( \frac{2\delta}{p} \right)^m \left( \frac{d(A)}{p} \right) \sqrt{|A|}, &  p \equiv 1 \: \textrm{mod} \: 4. 
\end{cases}
\end{equation*}  
Here $d(A)\in{\Fpu}/({\Fpu})^2$ is the discriminant of $A$, 
and $\delta\in{\Fp}$ is a square root of $-1$ which exists when $p\equiv 1$ mod $4$. 
When $A$ is the anisotropic plane, we have 
${\GAOA}=(-1)^{(p+1)/2}p$. 

$(3)$ Let $A=A_1\oplus A_2$ where $A_1\simeq{\Z}/p^k$ is a cyclic form as in $(1)$ with $k>1$ and $A_2$ is a $p$-elementary form as in $(2)$. 
Then 
\begin{equation*}\label{eqn:sec1 prod formula}
{\GAOA} = G(A_1, {\rm O}(A_1)) \cdot G(A_2, {\rm O}(A_2)). 
\end{equation*}
\end{theorem}

The product formula in $(3)$ is not trivial because ${\OA}$ does not preserve the decomposition $A=A_1\oplus A_2$ in general. 
When $p=2$, $k\leq3$, this formula does not hold. 
The formula in $(2)$ already shows that the most naive product formula does not hold. 

We were led to considering such an orbital exponential sum through the study of certain vector-valued modular forms. 
The Gauss sum ${\GAOA}$ appears in the dimension formula for them. 
We explain this in \S \ref{sec:dim formula}. 
In that formula also arises the following variant of ${\GAOA}$: 
\begin{equation*}
{\gaoa}= \sum_{[x]\in{\OA}\backslash A} e(-q(x))\sum_{y\in{\OA}x}e((x, y)). 
\end{equation*}
We show that a similar arithmetic formula holds for ${\gaoa}$. 
We state the result only for $p>3$, referring to \S \ref{sec:2nd kind} for the case $p=2, 3$. 

\begin{theorem}[\S \ref{sec:2nd kind}]
Let $p>3$. 

(1) If $A$ is a cyclic form as in Theorem \ref{main GAOA} (1), then 
\begin{equation*}
{\gaoa} = 
\frac{1}{2}\left( 1+ \left(\frac{p}{3}\right)^k \right)   \left( \frac{2a}{p}\right)^k  \sqrt{|A|} \times 
\begin{cases}
1, & p^k \equiv 1 \: \textrm{mod}\:  4,  \\ 
\sqrt{-1}, & p^k\equiv 3  \: \textrm{mod}\:  4. 
\end{cases}
\end{equation*}

(2) Let $A$ be a $p$-elementary form as in Theorem \ref{main GAOA} (2). 
When $A$ contains an isotropic vector, we have 
\begin{equation*}
{\gaoa} = 
\begin{cases}
0, & p\equiv 2 \: \textrm{mod}\:  3, \\
2\zeta_{16}^{m^2(p-1)^2} \left( \frac{2\delta}{p} \right)^m \left( \frac{(-1)^k d(A)}{p} \right) \sqrt{|A|}, & p\equiv 1 \: \textrm{mod}\:  3.  
\end{cases}
\end{equation*}
Here $\zeta_{16}=e(1/16)$, $k=[m/2]$ and $\delta\in{\Fp}$ is a primitive $6$-th root of $1$ which exists when $p\equiv 1$ mod $3$. 
When $A$ is the anisotropic plane, we have 
${\gaoa} = - \left( \frac{p}{3} \right) p$. 

(3) If $A=A_1\oplus A_2$ is a quadratic form as in Theorem \ref{main GAOA} (3), we have 
\begin{equation*}
{\gaoa} = G'(A_1, {\rm O}(A_1)) \cdot G'(A_2, {\rm O}(A_2)). 
\end{equation*}
\end{theorem}

One finds that there are many cases of vanishing for both ${\GAOA}$ and ${\gaoa}$. 
One also finds that the absolute values of ${\GAOA}$ and ${\gaoa}$ are either $0$ or $\sqrt{|A|}$ or $2\sqrt{|A|}$ 
in all cases we calculate. 

Our evaluation of ${\GAOA}$ and ${\gaoa}$ is done by direct calculation based on the explicit description of the ${\OA}$-orbits. 
This seems to be difficult for more general quadratic forms. 
The author does not know if there is more systematic way to evaluate ${\GAOA}$ and ${\gaoa}$, 
and whether a uniform arithmetic formula holds for general finite quadratic forms.


\section{Basic definitions}\label{sec:def}

A \textit{finite quadratic form} is a finite abelian group $A$ equipped with a ${\QZ}$-valued quadratic form $q: A \to {\QZ}$. 
This means that $q(ax)=a^2q(x)$ for $a\in{\Z}$ and $x\in A$, 
and that the ${\QZ}$-valued pairing 
\begin{equation*}
(x, y) = q(x+y) - q(x) - q(y), \quad x, y \in A, 
\end{equation*}
is symmetric bilinear. 
Unless stated otherwise, we assume that $(A, q)$ is nondegenerate, namely the bilinear form $(\, ,\, )$ is nondegenerate. 
We often abbreviate $A=(A, q)$. 
The orthogonal direct sum of two finite quadratic forms $A_1$, $A_2$ is written as $A_1\oplus A_2$. 
A map $A\to A$ is called isometry if it is an isomorphism of abelian groups and preserves $q$. 
The group of isometries of $(A, q)$ is denoted by ${\OA}={\rm O}(A, q)$ and called the orthogonal group of $(A, q)$. 
Finite quadratic forms are also called \textit{finite quadratic modules} in some literatures. 
A standard example is the discriminant form $L^{\vee}/L$ of an even lattice $L$, 
where the quadratic form $q$ is defined by $q(x+L)=(x, x)/2+{\Z}$ for $x\in L^{\vee}$. 
The associated bilinear form is the mod ${\Z}$ reduction of the pairing on the dual lattice $L^{\vee}$.  

Let $A_p$ be the $p$-component of $A$. 
The canonical decomposition $A=\oplus_{p}A_p$ as an abelian group is automatically an orthogonal decomposition. 
Indecomposable quadratic forms on $p$-groups are classified by Wall \cite{Wa}. 
When $p>2$, quadratic forms on $p$-groups can be reconstructed from the associated bilinear form by $q(x)=2^{-1}(x, x)$, 
as $2$ is invertible in ${\Zp}$. 
We mainly consider the bilinear form $(\; , \; )$ in this case. 
The $2$-adic case is more subtle. 

Let $\Gamma$ be a subgroup of ${\OA}$. 
We define the equivariant Gauss sum of $(A, \Gamma)$ as the exponential sum 
\begin{equation*}
G(A, \Gamma) = \sum_{[x]\in\Gamma\backslash A}\sum_{y\in \Gamma x} e((x, y)). 
\end{equation*}
This is well-defined: 
if we use $\gamma x$ in place of $x$ where $\gamma\in\Gamma$, then 
\begin{equation*}
\sum_{y\in \Gamma x} e((\gamma x, y)) = \sum_{y\in \Gamma x} e((x, \gamma^{-1}y)) = \sum_{y\in \Gamma x} e((x, y)). 
\end{equation*}
Similarly, we define the equivariant Gauss sum of the second kind by 
\begin{equation*}
G'(A, \Gamma) = \sum_{[x]\in\Gamma\backslash A} e(-q(x)) \sum_{y\in \Gamma x} e((x, y)). 
\end{equation*}
When $\Gamma$ is trivial,
\begin{equation*}
G(A,  \{ {\rm id}\}) = \sum_{x\in A}e((x, x)), \qquad G'(A,  \{ {\rm id}\}) = \sum_{x\in A}e(q(x)), 
\end{equation*}
are the classical quadratic Gauss sum. 
We will write $G(A)= G(A,  \{ {\rm id}\})$ and $G'(A)= G'(A,  \{ {\rm id}\})$. 
The evaluation of $G(A)$ and $G'(A)$ is well-known: see \cite{B-E-W}.  
We especially have the Milgram formula $G'(A)=e(\sigma(A)/8)\sqrt{|A|}$, 
where $\sigma(A)\in{\Z}/8{\Z}$ is the signature of $A$. 

Our object of study is the case $\Gamma={\OA}$. 
In the rest of this paper, for two ${\OA}$-orbits $[x], [y]\in {\OA}\backslash A$ we will write 
\begin{equation*}\label{orbital pairing}
\langle [x], [y] \rangle_{A} = \sum_{y'\in {\OA}y} e((x, y')). 
\end{equation*} 
(We often omit $A$ in the subscript.) 
This sum does not depend on the choice of an element $x$ from $[x]$, as checked above. 
In particular, since $-{\rm id}\in{\OA}$, we see that $\langle [x], [y] \rangle$ is a real number. 
Note that $\langle [x], [y] \rangle \ne \langle [y], [x] \rangle$ in general. 
The equivariant Gauss sums can be written as   
\begin{equation*}
{\GAOA} = \sum_{[x]\in {\OA}\backslash A} \langle [x], [x] \rangle_{A},  
\end{equation*}
\begin{equation*}
{\gaoa} = \sum_{[x]\in {\OA}\backslash A} e(-q(x)) \langle [x], [x] \rangle_{A}. 
\end{equation*}
Then ${\GAOA}$ is a real algebraic integer. 

We first localize the equivariant Gauss sum to each prime. 

\begin{lemma}\label{product decomposition Gauss sum}
Let $A$ be of the form $A=A_1\oplus A_2$ 
and assume that $\Gamma\subset {\OA}$ can be decomposed as 
$\Gamma=\Gamma_1\times \Gamma_2$ such that $\Gamma_i \subset {\rm O}(A_i)$. 
Then 
\begin{equation*}
G(A, \Gamma) = G(A_1, \Gamma_1) \cdot G(A_2, \Gamma_2), \qquad 
G'(A, \Gamma) = G'(A_1, \Gamma_1) \cdot G'(A_2, \Gamma_2). 
\end{equation*}
\end{lemma}

\begin{proof}
For $x=(x_1, x_2)\in A_1\oplus A_2$ the orbit $\Gamma x$ is decomposed as   $\Gamma x = (\Gamma_1x_1) \times (\Gamma_2x_2)$. 
Hence
\begin{eqnarray*}
G(A, \Gamma) 
&=& 
\sum_{[(x_1, x_2)]\in\Gamma\backslash A} \sum_{(y_1, y_2)\in \Gamma x} e((x_1, y_1)+(x_2, y_2)) \\ 
&=&  
\sum_{[x_1]\in\Gamma_1\backslash A_1} \sum_{[x_2]\in\Gamma_2\backslash A_2} \sum_{y_1\in\Gamma_1x_1} \sum_{y_2\in\Gamma_2x_2} 
e((x_1, y_1))\cdot e((x_2, y_2)) \\ 
&=& G(A_1, \Gamma_1) \cdot G(A_2, \Gamma_2). 
\end{eqnarray*} 
The case of $G'(A, \Gamma)$ is similar. 
\end{proof}

For $\Gamma={\OA}$ we have the canonical decomposition ${\OA}=\prod_{p}{\rm O}(A_p)$. 
Hence we obtain  

\begin{corollary}\label{localize Gauss sum}
We have 
\begin{equation*}
{\GAOA} = \prod_{p}G(A_p, {\rm O}(A_p)), \quad 
{\gaoa} = \prod_{p}G'(A_p, {\rm O}(A_p)). 
\end{equation*}
\end{corollary}

The evaluation of ${\GAOA}$ and ${\gaoa}$ is thus reduced to the case of quadratic forms on $p$-groups. 
For ${\GAOA}$ we study the case $p>2$ in \S \ref{sec:p odd} and the case $p=2$ in \S \ref{sec:p=2}. 
We study ${\gaoa}$ in \S \ref{sec:2nd kind}.


\section{Nondyadic case}\label{sec:p odd}

Let $p>2$. 
Let $A$ be an abelian $p$-group endowed with a ${\QZ}$-valued symmetric bilinear form $( \: , \: )$. 
In this section we evaluate ${\GAOA}$ for the following quadratic forms: 
(1) cyclic forms, (2) $p$-elementary forms (quadratic spaces over ${\Fp}$), 
and (3) direct sum $A=A_1\oplus A_2$ where $A_i$ is as in $(i)$. 
We write $A(\varepsilon)$ for the scaling of $A$ by $\varepsilon\in{\Z}_{p}^{\times}$. 
The symmetric bilinear form on ${\Z}/p^k$ defined by $(x, y)=axy/p^k$, $a\in{\Zpu}$, is denoted by $A_{p^k, a}$. 
We write $\left( \frac{\cdot}{p} \right)$ for the Legendre symbol. 
We also write $\zeta_{p^k}=e(1/p^k)$.

\subsection{Cyclic forms}\label{ssec:cyclic p>2}

\begin{proposition}
Let $A=A_{p^k, a}$ with $p>2$. 
Then  
\begin{equation*}
{\GAOA} = 
\begin{cases}
0, & p \equiv 3 \; (4), \; k \; \textrm{odd}, \\ 
\left( \frac{a}{p}\right)^k \sqrt{p^k}, & \textrm{otherwise}.
\end{cases}
\end{equation*}
\end{proposition}

\begin{proof}
The orthogonal group ${\OA}$ consists of $\pm{\rm id}$, and  $-{\rm id}$ fixes no nonzero element. 
Hence 
\begin{equation*}
{\GAOA} 
= 1 + \frac{1}{2}\sum_{\stackrel{x\in{\Z}/p^k}{x\ne0}}(\zeta_{p^k}^{ax^2}+\zeta_{p^k}^{-ax^2}) 
= {\rm Re}(G(A)). 
\end{equation*}
By \cite{B-E-W} Theorem 1.5.2 we have 
\begin{equation*}
G(A) = \left( \frac{a}{p^k}\right) \sqrt{p^k} \times 
\begin{cases}
1,            & p^k \equiv 1 \; (4),  \\ 
\sqrt{-1}, & p^k \equiv 3 \; (4).
\end{cases}
\end{equation*}
So ${\rm Re}(G(A))=0$ when $p^k \equiv 3\; (4)$, 
and ${\rm Re}(G(A))=\left( \frac{a}{p}\right)^k \sqrt{p^k}$ otherwise. 
\end{proof}

\subsection{$p$-elementary forms}

Let $A$ be a symmetric form on a $p$-elementary group with $p>2$. 
We can naturally view $A$ as a quadratic space over ${\Fp}$. 
By Witt's extension theorem (\cite{Ge} Theorem 2.44), 
two nonzero vectors of $A$ are ${\OA}$-equivalent if and only if they have the same norm. 
For $\mu\in{\Fp}$ let $A_{\mu}$ be the subset of $A$ of vectors $x$ such that $(x, x)=\mu$.  
Then 
\begin{equation*}
A = \{ 0 \} \cup (A_0\backslash\{ 0 \}) \cup \bigcup_{\mu\in{\Fpu}}A_{\mu} 
\end{equation*}
is the ${\OA}$-orbit decomposition of $A$. 

We first consider the case $A$ has a nonzero isotropic vector. 
In that case $A$ contains as a direct summand the \textit{hyperbolic plane} $U$, namely 
the symmetric form on ${\Fp}\oplus{\Fp}$ given by the Gram matrix 
$\begin{pmatrix} 0 & 1 \\ 1 & 0 \end{pmatrix}$. 
As the first step we calculate the case $A=U$. 

\begin{lemma}\label{claim:Gauss U}
We have 
$G(U, {\rm O}(U)) = (1+(-1)^{(p-1)/2})p$. 
\end{lemma}

\begin{proof}
Let $u_1, u_2$ be the standard basis of $U$. 
Then $U_0={\Fp}u_1 \cup {\Fp}u_2$, 
and for $\mu\ne0$ we have $U_{\mu} = \{ \nu u_1 + \nu^{-1}(2^{-1}\mu)u_2 \; | \;  \nu\in{\Fpu} \}$. 
As a reference vector in $U_{\mu}$ for $\mu\in{\Fp}$ we take 
\begin{equation*}\label{eqn:ref vec in U_mu}
x_{\mu} = u_1 + (2^{-1}\mu) u_2 \in U_{\mu}. 
\end{equation*}
We write for $\mu, \lambda \in{\Fp}$  
\begin{equation}\label{eqn:Gauss sum U-split 2}
\langle \mu, \lambda \rangle_U = \sum_{x\in U_{\lambda}} \zeta_{p}^{(x_{\mu}, x)} 
= \sum_{\nu\in{\Fpu}} \zeta_{p}^{2^{-1}(\mu\nu+\lambda\nu^{-1})}. 
\end{equation} 
We have 
\begin{equation*}
\langle \mu, \lambda \rangle_U = 
\begin{cases}
\langle U_{\mu}\backslash \{0\}, U_{\lambda} \rangle, & \lambda\ne0, \\ 
\langle U_{\mu}\backslash \{ 0 \}, U_{0}\backslash \{ 0 \} \rangle +1, & \lambda=0. 
\end{cases}
\end{equation*}
The second equality of \eqref{eqn:Gauss sum U-split 2} holds even when $\lambda=0$ 
(both sides equal to $-1$ when $\mu\ne0$, and to $p-1$ when $\mu=0$). 
This expression shows that $\langle \mu, \lambda \rangle_U=\langle \lambda, \mu \rangle_U$. 

By definition we have 
\begin{equation}\label{eqn: GUOU by def}
G(U, {\rm O}(U))  = \sum_{\mu\in{\Fp}} \langle \mu, \mu \rangle_{U} 
=\sum_{\nu\in{\Fpu}} \sum_{\mu\in{\Fp}} \zeta_p^{2^{-1}\mu(\nu+\nu^{-1})}. 
\end{equation}
When $p\equiv 1 \; (4)$, the equation $\nu+\nu^{-1}=0$ has two solutions, namely the square roots of $-1$, hence $G(U, {\rm O}(U)) = 2p$. 
When $p\equiv 3 \; (4)$, we have $\nu+\nu^{-1}\ne0$ for any $\nu\in{\Fpu}$. 
Thus $G(U, {\rm O}(U))=0$ in this case. 
\end{proof}

We consider the general case. 

\begin{proposition}\label{prop:Gauss U-split}
Let $A$ be an $m$-dimensional quadratic space over ${\Fp}$ with $p>2$ that contains a nonzero isotropic vector. 
Let $d(A)\in{\Fpu}/({\Fpu})^2$ be the discriminant of $A$. 
When $p\equiv 1 \; (4)$, we choose $\delta\in{\Fp}$ with $\delta^2=-1$. 
Then 
\begin{equation*}
{\GAOA} = 
\begin{cases} 
0, & p\equiv 3 \; (4), \\ 
2\left( \frac{2\delta}{p} \right)^m \left( \frac{d(A)}{p} \right) \sqrt{|A|}, & p\equiv 1 \; (4).  
\end{cases} 
\end{equation*}
\end{proposition}

\begin{proof}
We keep the notation in the proof of Lemma \ref{claim:Gauss U}. 
We choose a splitting $A=U\oplus B$, which gives the partition 
\begin{equation*}\label{eqn:partition U-split}
A_{\mu} = \bigsqcup_{\lambda\in{\Fp}} U_{\lambda}\times B_{\mu-\lambda}. 
\end{equation*}
As a reference vector in $A_{\mu}$ we use $(x_{\mu}, 0)\in U_{\mu}\times B_0$.  
Then 
\begin{equation}\label{eqn:Gauss sum U-split 1}
{\GAOA} = \sum_{\mu, \lambda \in {\Fp}} |B_{\mu-\lambda}| \cdot \langle \mu, \lambda \rangle_U 
\end{equation}
by this partition. 
The formula of $|B_{\alpha}|$ can be found in \cite{Ge} \S 2.8. 
It depends on the parity of $m={\dim}(A)$. 

(1) Let $m=2k$ be even. 
We write $\delta_A=\left( \frac{(-1)^kd(A)}{p}\right)$. 
Then $\delta_A=1$ if and only if $B\simeq U^{k-1}$. 
By \cite{Ge} Theorem 2.59 we have 
\begin{equation*}\label{eqn:Gauss sum U-split 3}
|B_{\alpha}| = 
\begin{cases}
p^{k-2}(p^{k-1}+\delta_{A}p-\delta_{A}) & \alpha=0, \\ 
p^{k-2}(p^{k-1}-\delta_A) & \alpha\ne0. 
\end{cases}
\end{equation*}
In particular, $|B_{\alpha}|$ for $\alpha\ne0$ is independent of $\alpha$ and hence 
\begin{equation}\label{eqn:Gauss sum U-split 3.5}
{\GAOA} = |B_0| \cdot G(U, {\rm O}(U))  + 
|B_1| \cdot \sum_{\stackrel{\mu, \lambda \in {\Fp}}{\mu\ne\lambda}}\langle \mu, \lambda \rangle_U. 
\end{equation}
Since 
\begin{equation}\label{eqn:sum over U =0}
\sum_{\mu, \lambda \in {\Fp}}\langle\mu, \lambda \rangle_U = \sum_{\mu\in{\Fp}} \sum_{x\in U} \zeta_p^{(x_{\mu}, x)} = 0, 
\end{equation}
we have 
\begin{equation*}
{\GAOA} = (|B_0|-|B_1|) \cdot G(U, {\rm O}(U)) 
= p^{k-1} \cdot \delta_A \cdot G(U, {\rm O}(U)), 
\end{equation*}
and the proposition follows from Lemma \ref{claim:Gauss U}.

(2) Let $m=2k+1$ be odd. 
We put $\tilde{d}(A)=(-1)^{k}d(A)$. 
Then $B\simeq U^{k-1}\oplus \langle \tilde{d}(A) \rangle$. 
By \cite{Ge} Theorem 2.60 we have for $\alpha\in{\Fp}$ 
\begin{equation*}
|B_{\alpha}| = p^{2k-2} + \left( \frac{\alpha\cdot\tilde{d}(A)}{p} \right) p^{k-1}. 
\end{equation*}
This depends on the class of $\alpha$ in ${\Fp}/({\Fpu})^2=\{ \bar{0}, \bar{1}, \bar{\varepsilon} \}$ where $\varepsilon\in{\Fpu}$ is a nonsquare. 
If we write 
\begin{equation*}
S = \{ (\mu, \lambda)\in{\Fp}\times{\Fp} \; | \; \mu-\lambda \in({\Fpu})^2 \}, 
\end{equation*}
\begin{equation*}
F = \sum_{(\mu, \lambda)\in S} \langle \mu, \lambda \rangle_U, 
\end{equation*}
we obtain from \eqref{eqn:Gauss sum U-split 1} and \eqref{eqn:sum over U =0}
\begin{eqnarray*}
{\GAOA} 
& = & 
|B_0|\cdot G(U, {\rm O}(U)) + |B_1|\cdot F + |B_{\varepsilon}|\cdot (-G(U, {\rm O}(U))-F) \\ 
& = & 
(|B_0|-|B_{\varepsilon}|)\cdot G(U, {\rm O}(U))  + (|B_1|-|B_{\varepsilon}|)\cdot F \\ 
& = & 
\left( \frac{\tilde{d}(A)}{p} \right) p^{k-1} \cdot G(U, {\rm O}(U))  + 2 \left( \frac{\tilde{d}(A)}{p} \right) p^{k-1} \cdot F. 
\end{eqnarray*}

We shall show that 
\begin{equation*}
F = 
\begin{cases}
0, & p \equiv 3 \; (4), \\
\left( \left( \frac{2\delta}{p} \right) \sqrt{p} - 1 \right) \cdot p, & p \equiv 1 \; (4), 
\end{cases}
\end{equation*}
from which the proposition follows. 
In case $p \equiv 3 \; (4)$, since $-1$ is nonsquare, switching $\mu$ and $\lambda$ gives 
\begin{equation*}
0 = \sum_{\mu, \lambda \in{\Fp}} \langle \mu, \lambda \rangle_U = 
G(U, {\rm O}(U))  + \sum_{(\mu, \lambda)\in S} \langle \mu, \lambda \rangle_U + \sum_{(\lambda, \mu)\in S} \langle \mu, \lambda \rangle_U = 2F. 
\end{equation*}
In case $p\equiv1\; (4)$, writing $\mu=\lambda+\alpha^2$ for $(\mu, \lambda)\in S$, we have 
\begin{eqnarray*}
2F  
& = & 
\sum_{\lambda\in{\Fp}} \sum_{\alpha\in{\Fpu}} \langle \lambda+\alpha^2, \lambda \rangle_U \\ 
& = & 
\sum_{\lambda\in{\Fp}} \sum_{\nu\in{\Fpu}} \zeta_p^{2^{-1}(\lambda\nu+\lambda\nu^{-1})} \sum_{\alpha\in{\Fpu}} \zeta_p^{2^{-1}\nu\alpha^2} \\ 
& = & 
\sum_{\nu\in{\Fpu}}  (G(A_{p, 2\nu})-1) \sum_{\lambda\in{\Fp}} \zeta_p^{2^{-1}\lambda(\nu+\nu^{-1})}. 
\end{eqnarray*}
The square roots $\pm\delta$ of $-1$ are the solutions of $\nu+\nu^{-1}=0$. 
We thus obtain 
\begin{equation*}
2F = 
2p (G(A_{p, 2\delta}) - 1) =  
2p  \left( \left( \frac{2\delta}{p} \right) \sqrt{p} - 1 \right). 
\end{equation*}
\end{proof}

It remains to consider the anisotropic case.  
Since the $1$-dimensional case is covered in \S \ref{ssec:cyclic p>2}, 
what remains is the anisotropic plane. 

\begin{proposition}\label{prop:Gauss anisotropic plane}
Let $A$ be the anisotropic plane over ${\Fp}$ with $p>2$. 
Then ${\GAOA} = (-1)^{(p+1)/2}p$. 
\end{proposition}

\begin{proof}
Fix a nonsquare $\varepsilon\in{\Fpu}$. 
Since the scaling $A(\varepsilon)$ is isometric to $A$ itself, 
there exists an isomorphism $j:A\to A$ of abelian groups such that $(j(x), j(y))=\varepsilon(x, y)$ for every $x, y\in A$. 
For $\lambda\in{\Fpu}$ we write 
\begin{equation*}
A_{\lambda}^{+}=\{ (\lambda, x) | x\in A_{\lambda^2} \} \simeq A_{\lambda^2}, \qquad 
A_{\lambda}^{-}=\{ (\lambda, x) | x\in A_{\varepsilon\lambda^2} \} \simeq A_{\varepsilon\lambda^2}, 
\end{equation*} 
and set 
\begin{equation*}
A^{\pm} = \bigsqcup_{\lambda\in{\Fpu}} A_{\lambda}^{\pm}, \qquad 
\tilde{A} = A^+ \sqcup A^-. 
\end{equation*}
$\tilde{A}$ is a double covering of $A\backslash\{0\}$, 
with the covering transformation $(\lambda, x)\mapsto(-\lambda, x)$ which switches $A_{\lambda}^{\pm}$ and $A_{-\lambda}^{\pm}$. 
On $\tilde{A}$ we have an ${\Fpu}$-action defined by $\alpha\cdot(\lambda, x)=(\alpha\lambda, \alpha x)$. 
We can choose a reference point $(\lambda, x_{\lambda}^{\pm})\in A_{\lambda}^{\pm}$ for each $\lambda$ so that 
$x_{\alpha\lambda}^{\pm}=\alpha x_{\lambda}^{\pm}$ and $x_{\lambda}^{-}=j(x_{\lambda}^{+})$. 
If we consider the mapping 
\begin{equation*}
\varphi : \tilde{A} \to {\Fp}, \qquad A_{\lambda}^{\pm}\ni (\lambda, x)\mapsto (x_{\lambda}^{\pm}, x), 
\end{equation*}
we have 
\begin{equation*}
2{\GAOA} = 2 + \sum_{(\lambda, x)\in\tilde{A}}\zeta_p^{\varphi(\lambda, x)}.
\end{equation*}

\begin{claim}\label{claim:constant fiber}
The fibers of $\varphi$ over ${\Fpu}\subset{\Fp}$ have constant cardinality. 
\end{claim}

\begin{proof}
We let ${\Fpu}$ act on ${\Fp}$ by weight $2$, i.e., $\alpha(t)=\alpha^2 t$. 
By our choice of $x_{\lambda}^{\pm}$, the map $\varphi$ is ${\Fpu}$-equivariant. 
Hence its fibers have constant cardinality over $({\Fpu})^2$ and over $\varepsilon({\Fpu})^2$ respectively. 
We shall show that 
\begin{equation}\label{eqn:equality of fiber degree}
|\varphi^{-1}(({\Fpu})^2)\cap A^{\pm}| = |\varphi^{-1}(\varepsilon({\Fpu})^2)\cap A^{\mp}|, 
\end{equation}
which would then imply $|\varphi^{-1}(({\Fpu})^2)| = |\varphi^{-1}(\varepsilon({\Fpu})^2)|$. 
We consider the map $j:A^+\to A^-$ defined by $(\lambda, x)\to(\lambda, j(x))$. 
Since $x_{\lambda}^{-}=j(x_{\lambda}^{+})$, 
we have $\varphi(j(\lambda, x))=\varepsilon\cdot\varphi(\lambda, x)$ for $(\lambda, x)\in A^+$. 
Hence 
\begin{equation*}
j(\varphi^{-1}(({\Fpu})^2)\cap A^{+}) \subset \varphi^{-1}(\varepsilon({\Fpu})^2)\cap A^{-},  
\end{equation*}
\begin{equation*}
j(\varphi^{-1}(\varepsilon({\Fpu})^2)\cap A^{+}) \subset  \varphi^{-1}(({\Fpu})^2)\cap A^{-},  
\end{equation*}
\begin{equation*}
j(\varphi^{-1}(0)\cap A^{+}) \subset  \varphi^{-1}(0)\cap A^{-}. 
\end{equation*}
Since $j:A^+\to A^-$ is bijective, the three inclusions are all equality.  
\end{proof}
 
By this claim we have 
\begin{eqnarray*}
2{\GAOA} 
& = & 
2 + |\varphi^{-1}(0)| + |\varphi^{-1}(1)|\sum_{\alpha\in{\Fpu}}\zeta_p^{\alpha} \\ 
& = & 
2 + |\varphi^{-1}(0)| - |\varphi^{-1}(1)| \\ 
& = & 
-2p + \frac{p}{p-1}|\varphi^{-1}(0)|. 
\end{eqnarray*}
In the last equality we used 
\begin{equation*} 
(p-1)|\varphi^{-1}(1)|+|\varphi^{-1}(0)|=|\tilde{A}|=2(p^2-1). 
\end{equation*}
We are thus reduced to calculating $|\varphi^{-1}(0)|$. 

For each $\lambda\ne0$, 
$\varphi^{-1}(0)\cap A_{\lambda}^{\pm}$ is identified with 
the set of vectors in $(x_{\lambda}^{\pm})^{\perp}\cap A$ having the same norm as $x_{\lambda}^{\pm}$. 
This set is non-empty if and only if $d(A)=-\varepsilon$ is a square, i.e., $\left( \frac{-1}{p} \right)=-1$. 
In that case it consists of two elements. 
It follows that  
\begin{equation*}
|\varphi^{-1}(0)| = \left( 1 - \left( \frac{-1}{p} \right) \right) \cdot 2(p-1). 
\end{equation*}
Therefore ${\GAOA}=-\left( \frac{-1}{p} \right) p$. 
\end{proof}

The above method can be extended to general $p$-elementary forms of even dimension. 
The result agrees with Proposition \ref{prop:Gauss U-split}, of course.

\subsection{Product formula}\label{ssec: product formula p>2}

Let $p>2$. 
Let $A$ be the direct sum $A=A_{p^k, a}\oplus B$ where $k>1$ and $B$ is $p$-elementary. 
We show that a product formula holds for ${\GAOA}$. 
This is not trivial as ${\OA}$ does not preserve the decomposition in general. 
We first describe the orthogonal group ${\OA}$, then classify the ${\OA}$-orbits, 
and finally calculate the Gauss sum. 

Let $e$ be the standard generator of $A_{p^k, a}\simeq {\Z}/p^k$. 
Using $e$, we can express an isomorphism $A\to A$ of abelian groups in the matrix form 
\begin{equation}\label{eqn:matrix form of A->A}
\begin{pmatrix} x & p^{k-1}f \\ v & g \end{pmatrix} 
\; \; \in \; \;  
\begin{pmatrix} ({\Z}/p^k)^{\times} & p^{k-1}{\hom}(B, {\Z}/p) \\ B & {\hom}(B, B) \end{pmatrix}. 
\end{equation} 
We define a subgroup of ${\OA}$ isomorphic to $B\rtimes {\OB}$ as follows. 
For $(v, g)\in B\rtimes {\OB}$ we define $x_v\in({\Z}/p^k)^{\times}$ and $f_{v,g}:B\to{\Z}/p$ by 
\begin{equation*}
x_v = 1-2^{-1}a^{-1}(p(v, v))p^{k-1}, 
\end{equation*}
\begin{equation*}
f_{v,g}(?) = -a^{-1}p(g^{-1}(v), ?), \quad ?\in B, 
\end{equation*}
where we view 
$p(v, v), p(g^{-1}(v), ?) \in {\Z}/p$. 
We define $\gamma_{v,g}:A\to A$ by 
\begin{equation*}\label{eqn:gamma(g,v)}
 \gamma_{v,g} = \begin{pmatrix} x_v & p^{k-1}f_{v,g} \\ v & g \end{pmatrix}. 
\end{equation*}
It is easy to check that $\gamma_{v,g}$ is an isometry and that the group  
\begin{equation*}\label{eeqn:O(B)ltimesB}
\Gamma = \{ \gamma_{v,g} \; | \; (v, g)\in B\rtimes {\OB} \}
\end{equation*}
is isomorphic to the semiproduct $B\rtimes {\OB}$. 

\begin{proposition}\label{prop:structure of OA}
We have ${\OA}=\{ \pm 1\} \times \Gamma$. 
\end{proposition}

\begin{proof}
The isometry condition for the matrix \eqref{eqn:matrix form of A->A} is 
\begin{equation}\label{eqn:isometry condition 1}
ax^2 + p^{k-1}(p(v, v)) \equiv a \mod p^k, 
\end{equation}
\begin{equation}\label{eqn:isometry condition 2}
a x f(?) + p(v, g(?)) = 0 \in {\hom}(B, {\Z}/p), \quad ?\in B, 
\end{equation}
\begin{equation*}\label{eqn:isometry condition 3}
g\in {\OB}. 
\end{equation*}
The solutions of \eqref{eqn:isometry condition 1} are explicitly given by $x=\pm x_v$. 
Then $f$ is uniquely determined from $v, x$ and $g$ by \eqref{eqn:isometry condition 2}.  
\end{proof}


We consider the following subsets of $A$:  
\begin{equation*}
A_0 = (({\Z}/p^k)^{\times}e) \times B, \qquad 
A_1 = (p({\Z}/p^{k-1})e) \times B. 
\end{equation*}
We have $A = A_0 \cup A_1$. 
Each $A_0$, $A_1$ is preserved by ${\OA}$.

\begin{lemma}\label{G-action on A_i p>2}
(1) The subset $(({\Z}/p^k)^{\times}e) \times \{ 0 \}$ of $A_0$ is a representative for $\Gamma\backslash A_0$. 
For $x\in({\Z}/p^k)^{\times}$ we have 
\begin{equation}\label{eqn:quasi-cyclic G-orbit}
\Gamma(xe) = \{ x(x_ve+v) \; | \; v\in B \}. 
\end{equation}
(2) The $\Gamma$-orbit of $xe+w\in A_1$, where $x\in p({\Z}/p^{k-1})$ and $w\in B$, is 
\begin{equation*}
\Gamma(xe+w) = 
\begin{cases}
\{ xe \}, & w=0, \\ 
(x + p^{k-1}({\Z}/p))e \times ({\OB}w), & w\ne0.  
\end{cases}
\end{equation*}
\end{lemma}
 
\begin{proof}
(1) Let $x\in({\Z}/p^k)^{\times}$. 
The description \eqref{eqn:quasi-cyclic G-orbit} of the orbit is apparent. 
It implies that $x'e\not\in \Gamma(xe)$ if $x\ne x' \in ({\Z}/p^k)^{\times}$. 
Since $|\Gamma(xe)|=|B|$, then $\Gamma(({\Z}/p^k)^{\times}e)=A_0$. 
The assertion (2) follows from 
\begin{equation*}
B(xe+w')=xe+B(w')=xe+w'+p^k(B, w')e
\end{equation*} 
for $x\in p({\Z}/p^{k-1})$ and $w'\in{\OB}w\subset B$. 
\end{proof}

Before calculating ${\GAOA}$, let us prepare a general formula. 

\begin{lemma}\label{redundant Gauss sum =0}
Let $p>2$, $a\in{\Zpu}$ and suppose that $k\geq2$. 
Then 
\begin{equation*}\label{eqn:a vanishing of exponential sum}
\sum_{x\in({\Z}/p^k)^{\times}} \zeta_{p^k}^{ax^2} = 0. 
\end{equation*}
\end{lemma}

\begin{proof}
If we write $S=(({\Z}/p^k)^{\times})^2\subset{\Z}/p^k$, 
the right hand side is written as $2\sum_{y\in S}\zeta_{p^{k}}^{ay}$. 
By the local square theorem, we have the additive action of $p({\Z}/p^{k-1})$ on $S$. 
On each orbit, say $S_i=y_i+p({\Z}/p^{k-1})$, we have 
\begin{equation*}
\sum_{y\in S_i} \zeta_{p^{k}}^{ay} = 
\sum_{y\in S_i} \zeta_{p^{k}}^{a(y+p)} = 
\zeta_{p^{k-1}}^{a} \cdot \sum_{y\in S_i} \zeta_{p^{k}}^{ay}. 
\end{equation*}
Since $\zeta_{p^{k-1}}^{a}\ne 1$, this sum is equal to $0$. 
\end{proof}

We are now ready to calculate the Gauss sum. 

\begin{proposition}\label{product formula general quasi-cyclic p>2}
Let $A$ be the direct sum $A=C\oplus B$ where $C=A_{p^k, a}$ with $p>2$, $k>1$ and $B$ is $p$-elementary. 
Then 
\begin{equation*}
{\GAOA} = G(C, {\rm O}(C)) \cdot G(B, {\rm O}(B)). 
\end{equation*}
\end{proposition}

\begin{proof}
We take the sum over each stratum $A_0, A_1$. 
We first consider $A_0$. 
By Lemma \ref{G-action on A_i p>2} (1), we can take $(({\Z}/p^k)^{\times}/-1)e$ as reference points of ${\OA}\backslash A_0$. 
For $x\in({\Z}/p^k)^{\times}$ we have by \eqref{eqn:quasi-cyclic G-orbit} 
\begin{eqnarray}\label{eqn:(3.10)}
\langle [xe], [xe] \rangle_{A} 
& = & 
\sum_{v\in B} (\zeta_{p^k}^{ax^2x_v} + \zeta_{p^k}^{-ax^2x_v}) \\ 
& = & 
2 {\rm Re}\left( \zeta_{p^k}^{ax^2}\cdot \sum_{v\in B}\zeta_{p}^{-2^{-1}x^2p(v, v)} \right) \nonumber \\ 
& = & 
2 {\rm Re}(\zeta_{p^k}^{ax^2}\cdot G(B(-2))). \nonumber 
\end{eqnarray}
Hence 
\begin{eqnarray*}
\sum_{[y]\in{\OA}\backslash A_0} \langle [y], [y] \rangle_{A} 
& = &  
\frac{1}{2} \sum_{x\in({\Z}/p^k)^{\times}}\langle [xe], [xe] \rangle_{A} \\ 
& = & 
{\rm Re}\left( G(B(-2)) \sum_{x\in({\Z}/p^k)^{\times}} \zeta_{p^k}^{ax^2} \right)
 = 0 
\end{eqnarray*}
by Lemma \ref{redundant Gauss sum =0}. 
We next consider the stratum $A_1$. 
For $xe+w\in A_1$ we have  
\begin{equation}\label{eqn: split of pairing at A1 stratum}
\langle [xe+w], [xe+w] \rangle_A = 
\begin{cases}
\langle [xe], [xe] \rangle_C \cdot \langle [w], [w] \rangle_B,  & w=0, \\ 
p \cdot \langle [xe], [xe] \rangle_C \cdot \langle [w], [w] \rangle_B,  & w\ne0, 
\end{cases}
\end{equation}
by Lemma \ref{G-action on A_i p>2} (2), 
where $[xe]\in C/-1$ and $[w]\in B/{\OB}$. 
Since 
$\langle [x'e], [x'e] \rangle_C = \langle [xe], [xe] \rangle_C$ for $x'\in x+p^{k-1}({\Z}/p)$, 
we have 
\begin{eqnarray*}
\sum_{[y]\in {\OA}\backslash A_1}\langle [y], [y] \rangle_{A} & = & 
\sum_{[xe]\in pC/-1}\langle [xe], [xe] \rangle_C \cdot \sum_{[w]\in {\OB}\backslash B}\langle [w], [w] \rangle_B \\ 
& = & G(C, {\rm O}(C)) \cdot G(B, {\rm O}(B)). 
\end{eqnarray*}
Here the second equality is a consequence of Lemma \ref{redundant Gauss sum =0}. 
\end{proof}


\section{$2$-adic case}\label{sec:p=2} 

In this section we study quadratic forms $A=(A, q)$ on $2$-groups. 
Let $(\: , \: ): A\times A\to {\QZ}$ be the associated bilinear form. 
For $x\in A$, $(x, x)=2q(x)$ is well-defined as an element of ${\Q}/2{\Z}$, not just of ${\QZ}$. 
We work with this refined symmetric bilinear form, from which the quadratic form $q$ can be recovered. 
This is one of the differences with the case $p>2$. 
Classification of quadratic forms on $2$-groups is more complicated (\cite{Wa}). 
Furthermore, the local square theorem is now $({\Z}_2^{\times})^2=1+8{\Z}_2$, 
hence ${\Z}_2^{\times}/({\Z}_2^{\times})^2=({\Z}/8)^{\times}$.
Note also that the square $x^2$ of $x\in{\Z}/2^k$ can be defined as an element of ${\Z}/2^{k+1}$ 
and hence $x^2/2$ is well-defined as an element of $\frac{1}{2}{\Z}/2^k$. 

For an odd number $a$ 
we write $A_{2^k, a}$ for the quadratic form $(x, y)=axy/2^k$ on ${\Z}/2^k$.  
($ax^2/2^k$ is considered as an element of $2^{-k}{\Z}/2{\Z}$ as noted above.)  
We denote by $U, V$ the quadratic forms on $({\Z}/2)^{\oplus2}$ expressed by the Gram matrices 
\begin{equation*}
\begin{pmatrix} 0 & 2^{-1} \\ 2^{-1} & 0 \end{pmatrix},  \quad 
\begin{pmatrix} 1 & 2^{-1} \\ 2^{-1} & 1 \end{pmatrix} \quad 
\textrm{mod} \; \; 
\begin{pmatrix} 2{\Z} & {\Z} \\ {\Z} & 2{\Z} \end{pmatrix}
\end{equation*}  
respectively.

\subsection{Cyclic forms}\label{ssec:cyclic p=2}
 
\begin{proposition}\label{cyclic p=2}
Let $A=A_{2^k, a}$. When $k=1$, we have ${\GAOA}=0$.  
When $k>1$, we have  
\begin{equation*}
{\GAOA} = \left( \frac{2}{a} \right)^k \sqrt{|A|}. 
\end{equation*}
\end{proposition}

\begin{proof}
The case $k=1$ is clear. 
Let $k>1$. 
We have ${\OA}=\{ \pm1 \}$, and $-1$ has two fixed points (the zero element and the unique order $2$ element). 
Then  
\begin{equation*}
2{\GAOA} = 
2 + 2 + \sum_{\stackrel{x\in {\Z}/2^k}{2x\ne0}} (\zeta_{2^k}^{ax^2} + \zeta_{2^k}^{-ax^2}) 
= 2{\rm Re}(G(A)). 
\end{equation*}
We have 
$G(A)=\left( \frac{2}{a} \right)^k (1+\sqrt{-1}^{a})\sqrt{2^k}$ 
by \cite{B-E-W} Proposition 1.5.3. 
\end{proof}

\subsection{$2$-elementary forms}\label{ssec:2-elementary}

Let $A$ be a quadratic form on a 2-elementary group. 
By the nondegeneracy there exists a unique element $x_A\in A$ such that $(x, x_A)=(x, x)$ mod ${\Z}$ for every $x\in A$. 
This element $x_A$ is called the \textit{characteristic element} of $A$. 
We denote $A_+=x_A^{\perp}$, the group of elements of $A$ whose norm is in ${\Z}$. 
When $x_A=0$, namely $A=A_+$, then $A$ is said to be \textit{special}. 
In that case we have $A\simeq U^m$ or $A\simeq U^{m-1}\oplus V$. 
We set $\delta_A$ by $\delta_A=1$ in the first case and $\delta_A=-1$ in the second case. 
When $A$ is non-special, it is isometric to 
$A'\oplus A''$ where $A'$ is special and $A''\simeq (A_{2,1})^i\oplus(A_{2,-1})^j$ with $1\leq i+j \leq 2$ (see \cite{Wa}). 
When $i+j=1$, $x_A$ is the generator of $A''$. 
When $i+j=2$, $x_A$ is the unique nonzero element of $A''$ with $(x_A, x_A)\in{\Z}$, and generates the radical of $A_+$. 

By definition the characteristic element $x_A$ is invariant under the action of ${\OA}$. 
Except this, description of ${\OA}$-orbits is the same as the nondyadic case. 
We use the following $2$-adic versions of Witt cancelation. 

\begin{lemma}\label{lem:Witt cancel 2adic}
$(1)$ Let $A$ be a $2$-elementary form and $A_1, A_2\subset A$ be nondegenerate special forms. 
If $A_1\simeq A_2$, then $A_1^{\perp}\simeq A_2^{\perp}$. 
In particular, there exists an isometry $A\to A$ which maps $A_1$ to $A_2$. 

$(2)$ Let $A_1, A_2$ be non-special $2$-elementary forms and let $a=1$ or $-1$. 
If $A_1\oplus A_{2,a} \simeq A_2\oplus A_{2,a}$, then $A_1\simeq A_2$. 
\end{lemma}

\begin{proof}
The assertion (1) is proved in \cite{M-S} Proposition 1.5. 
We consider (2) in case ${\dim}(A_1)={\dim}(A_2)$ is odd. 
Using (1) for $U^{k}\hookrightarrow A_i$, 
we may assume that $A_1$ and $A_2$ are one of 
\begin{equation*}
U\oplus A_{2,1}, \quad U\oplus A_{2,-1}, \quad V\oplus A_{2,1}, \quad V\oplus A_{2,-1}. 
\end{equation*}
When $A_1\not\simeq A_2$, we see that $A_1\oplus A_{2,a}$ and $A_2\oplus A_{2,a}$ have different representation numbers by direct calculation, 
hence cannot be isometric. 
The case ${\dim}(A_1)={\dim}(A_2)$ even is similar. 
\end{proof}

\begin{proposition}\label{lem:OA-equivalen 2-elementary}
Let $A$ be a finite quadratic form on a 2-elementary group and 
$x, y\in A$ be nonzero, non-characteristic elements. 
The elements $x$ and $y$ are ${\OA}$-equivalent if and only if they have the same norm in ${\Q}/2{\Z}$. 
\end{proposition}

\begin{proof}
It suffices to prove the "if" direction. 
We first consider the case $(x, x)\not\in{\Z}$. 
Since $\langle x \rangle$ and $\langle y \rangle$ are nondegenerate, 
we have the orthogonal decomposition 
$A=\langle x \rangle\oplus x^{\perp}= \langle y \rangle\oplus y^{\perp}$. 
Since $x$ and $y$ are non-characteristic, $x^{\perp}$ and $y^{\perp}$ are non-special. 
Then $x^{\perp}\simeq y^{\perp}$ by Lemma \ref{lem:Witt cancel 2adic} (2), so $x$ and $y$ are ${\OA}$-equivalent. 

Next let $(x, x)\in{\Z}$. 
We assume $A_+/{\rm rad}(A_+)\not\simeq V$: this exceptional case can be treated directly. 
We shall show that there exists an embedding $U\hookrightarrow A$ whose image contains $x$. 
The same also applies to $y$, and then our assertion follows from Lemma \ref{lem:Witt cancel 2adic} (1). 
Now since $x$ is nonzero and non-characteristic, it is not contained in the radical of $A_+$. 
Hence we can find an element $x'\in A_+$ such that $(x, x')=1/2$. 
If either $x$ or $x'$ has norm $0$, then $\langle x, x'\rangle \simeq U$. 
If both $x$ and $x'$ have norm $1$, then $\langle x, x'\rangle \simeq V$.  
By our assumption $\langle x, x'\rangle^{\perp}\cap A_+$ contains $U$ or $V$, 
so it contains a norm $1$ vector $x''$. 
Then $\langle x, x'+x''\rangle \simeq U$. 
\end{proof}

For $\mu\in 2^{-1}{\Z}/2{\Z}$ we write $A_{\mu}\subset A$ for the subset of vectors of norm $\mu$. 
Then the ${\OA}$-orbit decomposition of $A$ is 
\begin{equation*}
A = \{ 0 \} \cup \{ x_A \} \cup \bigcup_{\mu\in 2^{-1}{\Z}/2{\Z}} (A_{\mu}\backslash \{ 0, x_A \}). 
\end{equation*}
We first consider the case $A$ contains $U$. 

\begin{proposition}\label{isotropic 2-elementary}
Let $A$ be a 2-elementary form that contains $U$. 
Then 
\begin{equation*}
{\GAOA} = 
\begin{cases}
 \delta_A \sqrt{|A|}, & A: \: \textrm{special}, \\ 
0, & A: \; \textrm{non-special}. 
\end{cases}
\end{equation*}
\end{proposition}

\begin{proof}
We write $A=U\oplus B$ and let $u_1, u_2$ be the standard basis of $U$. 
In case $A$ is non-special, we can also write $A=A_{2,1}\oplus A_{2,-1}\oplus C$ because 
$U\oplus A_{2,\pm1}\simeq (A_{2,\pm1})^2\oplus A_{2,\mp1}$. 
We denote by $u_{\pm}$ the generator of $A_{2,\pm1}$. 
As a reference vector in $A_{\mu}$ we can take 
$x_{\mu}=u_1$, $u_1+u_2$, $u_{\pm}$ for $\mu=0, 1, \pm1/2$ respectively. 
We have the simple expression 
\begin{equation}\label{eqn:2-elemen proof 1}
{\GAOA} = \sum_{\mu\in2^{-1}{\Z}/2{\Z}} \sum_{x\in A_{\mu}} e((x_{\mu}, x)) 
\end{equation} 
because $(x_{\lambda}, x_A)=(x_A, x_A)$ for $\lambda=(x_A, x_A)$. 
Since we have the partitions 
\begin{equation*}
A_{\mu} = \{ 0, u_1, u_2 \} \times B_{\mu} \sqcup \{ u_1+u_2 \} \times B_{1-\mu}, \qquad \mu=0, 1, 
\end{equation*}
\begin{equation*}
A_{\pm1/2} = \{ u_{\pm} \} \times C_0 \sqcup \{ u_{\mp} \} \times C_1 \sqcup \{ 0, u_++u_- \} \times C_{\pm1/2}, 
\end{equation*}
we obtain 
\begin{equation}\label{eqn: isotropic 2-elementary1}
\sum_{x\in A_0} e((u_1, x)) = \sum_{x\in A_1} e((u_1+u_2, x)) = |B_0|-|B_1|, 
\end{equation}
\begin{equation}\label{eqn: isotropic 2-elementary2}
\sum_{x\in A_{\pm1/2}} e((u_{\pm}, x)) = |C_1|-|C_0|. 
\end{equation}
Similarly, we have 
\begin{equation*}
|A_0| = 3|B_0| + |B_1| = 2|C_0| + |C_{1/2}| + |C_{-1/2}|, 
\end{equation*}
\begin{equation*}
|A_1| = |B_0| + 3|B_1| = 2|C_1| + |C_{1/2}| + |C_{-1/2}|, 
\end{equation*}
and hence 
\begin{equation*}
|A_0| - |A_1| = 2 (|B_0| - |B_1|) = 2 (|C_0| - |C_1|). 
\end{equation*}
Hence ${\GAOA}=0$ when $A$ is non-special, 
and ${\GAOA}=|A_0| - |A_1|$ when $A$ is special. 
In the latter case we have $|A_0| - |A_1|=\delta_A \sqrt{|A|}$ by induction on ${\dim}(A)/2$.   
\end{proof}

The remaining cases are covered by the following. 

\begin{proposition}\label{anisotropic 2-elementary}
We have 
\begin{equation*}
{\GAOA} = 
\begin{cases}
0, & A = V, \; A_{2,1}\oplus A_{2,-1}, \; (A_{2,a})^{3}, \\ 
\sqrt{|A|}, & A = (A_{2,a})^2, \; (A_{2,a})^4. 
\end{cases}
\end{equation*}
\end{proposition}

\begin{proof}
This can be checked directly. 
\end{proof}

\subsection{Product formula}\label{ssec: product formula p=2}

Let $A$ be the direct sum $A=A_{2^k,a}\oplus B$ such that $k\geq2$ and $B$ is $2$-elementary. 
We show that a product formula like Proposition \ref{product formula general quasi-cyclic p>2} holds when $k\geq4$, 
but not always when $k\leq3$. 
As in \eqref{eqn:matrix form of A->A}, we express an isomorphism $A\to A$ in the form 
$\begin{pmatrix} x & 2^{k-1}f \\ v & g \end{pmatrix}$. 
The isometry condition is 
\begin{equation}\label{eqn:isometry condition 1 p=2}
x^2 \equiv 1 - 2^{k-1}\cdot a (2(v, v))   \mod 2^{k+1}, 
\end{equation}
\begin{equation}\label{eqn:isometry condition 2 p=2}
f(?) \equiv 2(v, g(?)) \mod 2, \quad ?\in B, 
\end{equation}
\begin{equation}\label{eqn:isometry condition 3 p=2}
( g(?), g(?) ) \equiv (?, ?) - f(?) \cdot 2^{k-2} \mod 2{\Z}, \quad ?\in B, 
\end{equation}
where 
$2(v, v) \in {\Z}/4$ and $2(v, g(?)), f(?)\in{\Z}/2$. 
The condition \eqref{eqn:isometry condition 3 p=2} becomes $g\in {\rm O}(B)$ when $k\geq3$. 
When $k\geq 4$, the equation \eqref{eqn:isometry condition 1 p=2} always has two solutions,  
one of which is given by  
\begin{equation*}
x_v = 
\begin{cases}
1 - 2^{k-2}a(2(v, v)), & k\geq 5, \\ 
1 + 4a(2(v, v)), & k=4. 
\end{cases}
\end{equation*}
In case $k=3$, by the local square theorem, \eqref{eqn:isometry condition 1 p=2} has a solution if and only if $(v, v)\in{\Z}$. 
In that case, a solution is given by $x_v=1+4(v, v)$. 

Let $k\geq4$. 
For $(v, g)\in B\rtimes{\OB}$ we define the isometry $\gamma_{v,g}$ of $A$ by 
$\gamma_{v,g} = \begin{pmatrix} x_v & 2^{k-1}f_{v,g} \\ v & g \end{pmatrix}$ 
where $f_{v,g}$ is defined from $v$ and $g$ by \eqref{eqn:isometry condition 2 p=2}. 
We put 
\begin{equation}\label{eqn:structure of OA p=2}
\Gamma =\{ \gamma_{v,g} \: | \: (v, g)\in B\rtimes{\OB} \}. 
\end{equation}
We have ${\OA}=\{ \pm 1\} \times \Gamma$ by the same argument as in Proposition \ref{prop:structure of OA}, 
and the assertion of Lemma \ref{G-action on A_i p>2} still holds.  
Also Lemma \ref{redundant Gauss sum =0} holds in $k\geq4$ by the local square theorem in $p=2$: 
\begin{equation*}
\sum_{x\in({\Z}/2^k)^{\times}} \zeta_{2^k}^{ax^2} = 0, \qquad k\geq4. 
\end{equation*} 
Hence the argument of \S \ref{ssec: product formula p>2} works in $p=2$, $k\geq4$:

\begin{proposition}\label{general quasi-cyclic p=2}
Let $A$ be the direct sum $A=C\oplus B$ such that 
$C=A_{2^k,a}$ with $k\geq4$ and $B$ is $2$-elementary.  
Then   
\begin{equation*}\label{eqn:product formula Gauss p=2}
{\GAOA} = G(C, {\rm O}(C)) \cdot G(B, {\rm O}(B)). 
\end{equation*}
\end{proposition}

We consider the case $k=2, 3$. 

\begin{proposition}\label{prop:(p,k)=(2,3)}
Let $A=C\oplus B$ where $C=A_{8,a}$ and $B$ is $2$-elementary. Then
\begin{equation*}
{\GAOA} = 
\begin{cases}
0, &   B: \; \textrm{non-special}, \\ 
G(C, {\rm O}(C)) \cdot G(B, {\rm O}(B)), &  B: \; \textrm{special}, \\ 
-2G(C, {\rm O}(C)), & B=V. 
\end{cases} 
\end{equation*}
\end{proposition}

By \S \ref{ssec:2-elementary}, 
the product formula does not hold when $B=V$, $(A_{2,a})^2$, $(A_{2,a})^4$. 

\begin{proof}
We replace \eqref{eqn:structure of OA p=2} by the group $\Gamma_+ = B_+\rtimes {\rm O}(B)$ 
where $B_+\subset B$ is the subgroup of elements of norm $\in{\Z}$. 
Then ${\OA}=\{ \pm 1\} \times \Gamma_+$. 
Let 
\begin{equation*}
A_{0+}=(({\Z}/8)^{\times}e)\times B_+, \quad 
A_{0-}=(({\Z}/8)^{\times}e)\times (B\backslash B_+), \quad 
A_1=(2({\Z}/4)e)\times B, 
\end{equation*}
which are preserved by ${\OA}$. 
If $xe+w\in A_1$, then 
\begin{equation*}
{\OA}(xe+w) = (\pm x + 4(2(B_+, w)))e \times {\rm O}(B)w, 
\end{equation*}
where $2(B_+, w)$ is a subgroup of ${\Z}/2{\Z}$. 
Hence we have 
\begin{equation*}
\sum_{[y]\in{\OA}\backslash A_1} \langle [y], [y] \rangle_A 
= \left( \sum_{z\in2C/-1} \langle [z], [z] \rangle_C \right) \cdot G(B, {\rm O}(B)) = 0,  
\end{equation*}
where $\sum_{2C/-1} \langle [z], [z] \rangle_C=0$ by direct calculation. 
For $A_{0+}$, the coset ${\OA}\backslash A_{0+}$ consists of $[e]$ and $[3e]$. 
Since  
\begin{equation*}
{\OA}(xe) = \{ \pm (x+4(v, v))e+v \; | \; v\in B_+ \}  
\end{equation*}
for $x=1, 3$, we obtain  
\begin{equation*}
\sum_{[y]\in{\OA}\backslash A_{0+}} \langle [y], [y] \rangle 
= 2(\zeta_{8}^{a}+\zeta_{8}^{-a})\sum_{v\in B_+}(-1)^{(v, v)} 
= G(C, {\rm O}(C))\cdot (|B_0|-|B_1|). 
\end{equation*}
As shown in the proof of Proposition \ref{isotropic 2-elementary}, we have 
$|B_0|-|B_1|=G(B, {\rm O}(B))$ when $B$ is special and $B\ne V$, which proves the proposition in this case. 
The case $B=V$ is immediate. 
Suppose that $B$ is non-special. 
If we choose $w_0\in B\backslash B_+$, then 
${\OA}\backslash A_{0-}$ consists of $[e+w_0]$ and $[3e+w_0]$. 
For $x=1, 3$ we have 
\begin{equation*}
{\OA}(xe+w_0) = \{ \pm(x+4((w, w)-(w_0, w_0))) e + w \: | \: w\in B\backslash B_+ \}. 
\end{equation*}
It follows that 
\begin{eqnarray*}
\langle [xe+w_0], [xe+w_0] \rangle_A 
& = & 
(\zeta_8^a+\zeta_8^{-a}) \sum_{w\in B\backslash B_+} (-1)^{(w, w)-(w_0, w_0)+2(w_0, w)} \\ 
& = & 
-(\zeta_8^a+\zeta_8^{-a}) \sum_{v\in B_+} (-1)^{(v, v)} 
= 
-\langle [xe], [xe] \rangle_A. 
\end{eqnarray*}
Therefore ${\GAOA}=0$ in this case. 
\end{proof}

\begin{proposition}\label{prop:(p,k)=(2,2)}
Let $A=A_{4,a}\oplus B$ where $B$ is $2$-elementary. Then
\begin{equation*}
{\GAOA} = 
\begin{cases}
4, &   B=(A_{2,a})^{2}, \; A_{2,1}\oplus A_{2,-1}, \\ 
0, &  otherwise. 
\end{cases} 
\end{equation*}
\end{proposition}

\begin{proof}
If $x\in A$ is of order $4$, then for all $y\in{\OA}x$ we have $(x, y) \equiv \pm1/4$ mod ${\Z}$ and $y\ne -y$. 
Then 
\begin{equation*}
\langle [x], [x] \rangle_A = \sum_{y\in{\OA}x/\pm1} e((x, y)) + e((x, -y)) = \sum_y (\sqrt{-1}-\sqrt{-1}) = 0, 
\end{equation*}
hence ${\GAOA}$ reduces to 
\begin{equation*}\label{eqn: prod (p,k)=(2,2) 1}
{\GAOA} = \sum_{\stackrel{[x]\in{\OA}\backslash A}{2x=0}} \langle [x], [x] \rangle_A. 
\end{equation*}

When $B=A_{2,1}\oplus A_{2,-1}$, the assertion follows by direct calculation. 
When $B$ is anisotropic, the decomposition $A=A_{4,a}\oplus B$ is canonical, 
so the assertion follows from Lemma \ref{product decomposition Gauss sum} and 
the results of \S \ref{ssec:cyclic p=2} and \S \ref{ssec:2-elementary}. 
We consider the general case $B=B'\oplus B''$ where $B'\ne0$ is special and 
$B''=(A_{2,1})^i \oplus (A_{2,-1})^j$ with $0\leq i+j\leq2$. 
We will show that ${\GAOA}=0$ in this case. 
Let $x_B\in B''$ be the characteristic element of $B$. 
Let $e\in A_{4,a}$ be a generator. 
We have the exceptional elements $0$, $2e$, $x_B$, $2e+x_B$ fixed by ${\OA}$. 
($x_B$ can be $0$.) 
Extending Lemma \ref{lem:OA-equivalen 2-elementary}, 
we can show that other two elements of order $2$ in $A$ are ${\OA}$-equivalent if and only if they have the same norm. 
The possibilities for $x_B$ are: 
(1) $x_B=0$, 
(2) $x_B\ne0$, $(x_B, x_B)=0$, 
(3) $(x_B, x_B)=1$, and 
(4) $(x_B, x_B)=\pm1/2$. 

We consider only the case (2): other cases can be calculated similarly. 
If $x\in A$ has norm $0$ with $x\ne0, x_B$, then 
\begin{equation*}
{\OA}x = \{ y\in B \: | \: (y, y)=0, y\ne0, x_B \} \sqcup  \{ y+2e \: | \: y\in B, (y, y)=1 \}. 
\end{equation*}
Hence 
\begin{equation}\label{eqn: prod (p,k)=(2,2) 2}
\langle [x], [x] \rangle_A = \sum_{\stackrel{y\in B+}{y\ne0, x_B}} e((x, y)) = -2. 
\end{equation}
Similarly, for other orbits $[x]$, we have 
\begin{equation}\label{eqn: prod (p,k)=(2,2) 3}
\langle [x], [x] \rangle_A = 
\begin{cases}
-2, & (x, x)=1, \: x\ne2e, 2e+x_B, \\ 
0,  & (x, x)=\pm1/2. 
\end{cases}
\end{equation}
The contribution from the exceptional orbits $0$, $2e$, $x_B$, $2e+x_B$ is 
$1+1+1+1=4$, hence ${\GAOA}=0$. 
\end{proof}


\section{Equivariant Gauss sum of the second kind}\label{sec:2nd kind}

In this section we evaluate the equivariant Gauss sum ${\gaoa}$ of the second kind 
for quadratic forms $A$ as in the previous sections. 
Some part is similar to the case of ${\GAOA}$, so we omit the detail there. 

\subsection{Cyclic forms}

\begin{proposition}\label{prop: cyclic 2nd kind}
Let $A=A_{p^k,a}$ with $p\geq2$ and $a\in{\Zpu}/({\Zpu})^2$. 

$(1)$ If $p>3$, then 
\begin{equation*}
{\gaoa} = 
\frac{1}{2}\left( 1+ \left(\frac{p}{3}\right)^k \right)   \left( \frac{2a}{p}\right)^k  \sqrt{|A|} \times 
\begin{cases}
1, & p^k\equiv 1 \; (4),  \\ 
\sqrt{-1}, & p^k\equiv 3 \; (4). 
\end{cases}
\end{equation*}

$(2)$ For $p=2, 3$ we have  
\begin{equation*}
{\gaoa} = 
\begin{cases}
-\sqrt{-1}^{k^2} \left(\frac{-a}{3}\right)^k \zeta_3^{-a}  \sqrt{|A|}, &  p=3, \\  
\frac{1}{2} (1+(-1)^{k+1}) \left( \frac{2}{a} \right)^{k+1} \zeta_{8}^{-1} \zeta_4^{(a+1)^2/4}  \sqrt{|A|}, & p=2.  
\end{cases}
\end{equation*}
\end{proposition}

\begin{proof}
When $p>2$, we have
\begin{eqnarray*}
{\gaoa} & = & 
1 + \frac{1}{2}\sum_{\begin{subarray}{c} x \in{\Z}/p^k \\ x\ne0 \end{subarray}} 
\zeta_{p^k}^{-2^{-1}ax^2} (\zeta_{p^k}^{ax^2}+\zeta_{p^k}^{-ax^2})  \\ 
& = & \frac{1}{2} \{ G(A(2)) + G(A(-6)) \}. 
\end{eqnarray*}
When $p>3$, the assertion follows from the formula of $G(A(\varepsilon))$ (\cite{B-E-W} Theorem 1.5.2) and 
the quadratic reciprocity $(\frac{-3}{p})=(\frac{p}{3})$. 
When $p=3$, $A(-6)\simeq A(3)$ is degenerate with kernel $3^{k-1}A\simeq{\Z}/3$, so we have  
\begin{eqnarray*}
{\gaoa} 
& = & 
\frac{1}{2}  \left( G(A(-1)) + 3G(A_{3^{k-1},a})  \right)  \\ 
& = & 
\frac{1}{2} \cdot \left( \frac{-a}{3} \right)^k \cdot \sqrt{3^k} \cdot 
\left\{ \sqrt{-1}^{k^2} +  \sqrt{3} \left( \frac{-1}{3} \right)^k \left( \frac{a}{3} \right)\sqrt{-1}^{(k-1)^2} \right\}   \\ 
& = & 
\left( \frac{-a}{3} \right)^k \sqrt{-1}^{k^2} \sqrt{3^k} \cdot \frac{1}{2}\left( 1 + \left( \frac{a}{3} \right)\sqrt{-3} \right). 
\end{eqnarray*}

When $p=2$, $k\geq3$, we have  
\begin{eqnarray*}
{\gaoa} 
& = & 
1 + 1 + \frac{1}{2} \sum_{\begin{subarray}{c} x \in{\Z}/2^k \\ 2x\ne0 \end{subarray}} (\zeta_{2^{k+1}}^{ax^2} + \zeta_{2^{k+1}}^{-3ax^2}) \\ 
& = & 
\frac{1}{4} \{ G(A_{2^{k+1}, a}) + G(A_{2^{k+1}, -3a}) \} \\ 
& = & 
\frac{1}{4} \left(   \left( \frac{2}{a} \right)^{k+1} +  \left( \frac{2}{-3a} \right)^{k+1} \right) \cdot 
(1+\sqrt{-1}^{a}) \cdot \sqrt{2^{k+1}}. 
\end{eqnarray*}
The case $p=2$, $k\leq2$ can be calculated directly. 
\end{proof}

\subsection{$p$-elementary forms}

We first consider the case $p>2$. 
We begin with the hyperbolic plane $U$. 

\begin{lemma}\label{eqn: guou}
Let $p>2$. 
We have $G'(U, {\rm O}(U)) = \left( 1+\left( \frac{p}{3} \right) \right) p$. 
\end{lemma}

\begin{proof}
As in \eqref{eqn: GUOU by def} we have 
\begin{equation*}
G'(U, {\rm O}(U)) = 
\sum_{\mu\in{\Fp}} \zeta_p^{-2^{-1}\mu} \langle \mu, \mu \rangle_U = 
\sum_{\nu\in{\Fpu}}\sum_{\mu\in{\Fp}} \zeta_p^{2^{-1}\mu(-1+\nu+\nu^{-1})}. 
\end{equation*}
When $p\ne3$, the equation $-1+\nu+\nu^{-1}=0$ has (two) solutions in ${\Fpu}$ if and only if $\left( \frac{-3}{p} \right) =1$, 
and we have $\left( \frac{-3}{p} \right) = \left( \frac{p}{3} \right)$. 
When $p=3$, this equation has the unique solution $\nu=-1$. 
\end{proof}

\begin{proposition}\label{p-elementary 2nd kind}
Let $A$ be a quadratic space over ${\Fp}$ with $p>2$ of dimension $m$ and discriminant $d(A)$. 
We write $k=[m/2]$. 
When $p\equiv 1$ mod $3$, choose a solution $\delta\in{\Fp}$ of $\delta^2-\delta+1=0$. 
($\delta$ is a primitive $6$-th root of $1$ in ${\Fp}$.)

1) When $A$ contains an isotropic vector, then 
\begin{equation*}
{\gaoa} = 
\begin{cases}
0, & p\equiv 2 \; (3), \\
2\zeta_{16}^{m^2(p-1)^2} \left( \frac{2\delta}{p} \right)^m \left( \frac{(-1)^k d(A)}{p} \right) \sqrt{|A|}, & p\equiv 1 \; (3), \\ 
(
\sqrt{-1}^{-m}  \left( \frac{d(A)}{3} \right)   \sqrt{|A|}, & p=3.  
\end{cases}
\end{equation*}

(2) When $A$ is the anisotropic plane, then ${\gaoa} = - \left( \frac{p}{3} \right) p$. 
\end{proposition}

\begin{proof}
(1) We reuse the notation in the proof of Proposition \ref{prop:Gauss U-split}. 
When $m=2k$ is even, we have as in \eqref{eqn:Gauss sum U-split 1} and \eqref{eqn:sum over U =0}
\begin{eqnarray*}
{\gaoa} 
& = & 
\sum_{\mu, \lambda \in {\Fp}} \zeta_{p}^{-2^{-1}\mu} \cdot |B_{\mu-\lambda}| \cdot \langle \mu, \lambda \rangle_U \\  
& = & 
|B_0|\cdot G'(U, {\rm O}(U)) + 
|B_1|\cdot \sum_{\mu\ne\lambda }\zeta_{p}^{-2^{-1}\mu}\cdot\langle \mu, \lambda \rangle_{U} \\ 
& = & 
(|B_0|-|B_1|) \cdot G'(U, {\rm O}(U)) \\ 
& = & 
\delta_A \cdot p^{k-1} \cdot G'(U, {\rm O}(U)). 
\end{eqnarray*}
In the third equality we used 
$\sum_{\mu, \lambda}\zeta_{p}^{-2^{-1}\mu}\langle \mu, \lambda \rangle_{U}=0$.  
The proposition follows from Lemma \ref{eqn: guou}. 

Next let $m=2k+1$ be odd. 
We set 
\begin{equation*}\label{eqn:gauss sum U-split 1}
F' = \sum_{(\mu, \lambda)\in S} \zeta_{p}^{-2^{-1}\mu} \langle \mu, \lambda \rangle_U. 
\end{equation*}
Then ${\gaoa}$ is given by   
\begin{eqnarray*}\label{eqn:gauss sum U-split 0}
& & 
|B_0|\cdot G'(U, {\rm O}(U)) + |B_1| \cdot F' - |B_{\varepsilon}|\cdot (G'(U, {\rm O}(U))+F') \\ 
& = & 
\left( \frac{\tilde{d}(A)}{p} \right) p^{k-1} \cdot G'(U, {\rm O}(U)) + 2 \left( \frac{\tilde{d}(A)}{p} \right) p^{k-1} \cdot F'. 
\end{eqnarray*}
So it suffices to calculate $F'$.   
We have 
\begin{eqnarray*}
2F'  
& = & 
\sum_{\mu\in{\Fp}}  \sum_{\alpha\in{\Fpu}}  \sum_{\nu\in{\Fpu}}
\zeta_{p}^{2^{-1}\mu(-1+\nu+\nu^{-1})-2^{-1}\alpha^2\nu^{-1}} \\ 
& = & 
\sum_{\nu\in{\Fpu}} \left( \sum_{\mu\in{\Fp}} \zeta_{p}^{2^{-1}\mu(-1+\nu+\nu^{-1})} \right) \cdot (G(A_{p, -2\nu})-1). 
\end{eqnarray*}
When $p\equiv 2 \; (3)$, we have $\nu+\nu^{-1}\ne1$ for any $\nu\in{\Fpu}$ so that $F'=0$. 
When $p\equiv 1 \; (3)$, we have two solutions $\delta, \delta^{-1}$ of $\nu+\nu^{-1}=1$. 
Then $F'=p(G(A_{p,-2\delta})-1)$.  
Finally, when $p=3$, $\nu+\nu^{-1}=1$ has the unique solution $\nu=-1$ and thus 
$2F'=3(G(A_{3,-1})-1)$. 
It is now straightforward to finish the calculation of ${\gaoa}$. 

(2) We keep the notation in the proof of Proposition \ref{prop:Gauss anisotropic plane}. 
Instead of $\varphi$, we consider the map 
\begin{equation*}
\varphi' : \tilde{A} \to {\Fp}, \quad A_{\lambda}^{\pm}\ni (\lambda, x)\mapsto (x_{\lambda}^{\pm}, x) - (x_{\lambda}^{\pm}, x_{\lambda}^{\pm})/2. 
\end{equation*}
Then 
\begin{equation*}
2{\gaoa} = 2 + \sum_{(\lambda, x)\in\tilde{A}}\zeta_{p}^{\varphi'(\lambda, x)}. 
\end{equation*}
As in Claim \ref{claim:constant fiber}, $\varphi'$ has constant fiber cardinality over ${\Fpu}\subset{\Fp}$. 
Hence 
\begin{equation*}
2{\gaoa}  =  
2 + |(\varphi')^{-1}(0)| - |(\varphi')^{-1}(1)|  
= -2p + \frac{p}{p-1}|(\varphi')^{-1}(0)|. 
\end{equation*}
If $\varphi'(\lambda, x)=0$, we can write $x=(x_{\lambda}^{\pm}+y)/2$ for some $y\in(x_{\lambda}^{\pm})^{\perp}$. 
Hence 
\begin{equation*}
(\varphi')^{-1}(0)\cap A_{\lambda}^{\pm} \simeq  
\{ \: y\in(x_{\lambda}^{\pm})^{\perp} \: | \: (y, y) = 3(x_{\lambda}^{\pm}, x_{\lambda}^{\pm}) \: \}. 
\end{equation*}
When $p>3$, this set is non-empty (consisting of two points) if and only if 
$3\in d(A)\cdot ({\Fpu})^2$, namely $-3$ is nonsquare. 
It follows that 
\begin{equation*}
|(\varphi')^{-1}(0)| = \left( 1 - \left( \frac{p}{3} \right) \right) \cdot 2(p-1). 
\end{equation*}
When $p=3$, $(\varphi')^{-1}(0)\cap A_{\lambda}^{\pm}$ consists of one point 
and hence $|(\varphi')^{-1}(0)|=2(p-1)$. 
\end{proof}

\begin{proposition}
Let $A$ be a 2-elementary form of dimension $m>1$. 
Then 
\begin{equation*}
{\gaoa} = 
\begin{cases} 
0, & A \; \textrm{contains} \; U \; \textrm{or} \;  A=(A_{2,a})^2, \\
\sqrt{|A|}, & A=V, \: A_{2,1}\oplus A_{2,-1}, \; (A_{2,a})^4, \\
2(1+\sqrt{-1}^{-a}), & A=(A_{2,a})^3.  
\end{cases}
\end{equation*}
\end{proposition}

\begin{proof}
The exceptional cases can be checked directly. 
Let $A$ contain $U$. 
We have as in \eqref{eqn:2-elemen proof 1}
\begin{equation*}
{\gaoa} = \sum_{\mu\in2^{-1}{\Z}/2{\Z}} \sqrt{-1}^{-2\mu} \sum_{x\in A_{\mu}} e((x_{\mu}, x)). 
\end{equation*} 
By \eqref{eqn: isotropic 2-elementary1} and \eqref{eqn: isotropic 2-elementary2} 
we obtain ${\gaoa}=0$. 
\end{proof}

\subsection{Product formula}

\begin{proposition}
Let $p\geq2$ and $A$ be the direct sum $A=C\oplus B$ where $C=A_{p^k, a}$ with $k>1$ and $B$ is $p$-elementary. 
Then 
\begin{equation*}
{\gaoa} = G'(C, {\rm O}(C)) \cdot G'(B, {\rm O}(B)).  
\end{equation*}
\end{proposition}

\begin{proof}
We first consider the case $p>2$. 
We use the notation in \S \ref{ssec: product formula p>2}.  
For $[xe]\in {\OA}\backslash A_0$, $x\in({\Z}/p^k)^{\times}$, we have by \eqref{eqn:(3.10)}  
\begin{equation*}
e(-q(xe)) \langle [xe], [xe] \rangle_{A} = 
\zeta_{p^k}^{2^{-1}ax^2}G(B(-2)) + \zeta_{p^k}^{-2^{-1}\cdot3ax^2}G(B(2)). 
\end{equation*}
If $p>3$, we obtain   
\begin{equation*}
\sum_{[y]\in {\OA}\backslash A_0} e(-q(y)) \langle [y], [y] \rangle_A = 0. 
\end{equation*}
by Lemma \ref{redundant Gauss sum =0}. 
When $p=3$, we have 
\begin{equation*}
\sum_{[y]\in {\OA}\backslash A_0} e(-q(y)) \langle [y], [y] \rangle_A =
G(B(2))\cdot \frac{3}{2} \sum_{x\in({\Z}/3^{k-1})^{\times}}\zeta_{3^{k-1}}^{-2^{-1}ax^2}. 
\end{equation*}
This is equal to $0$ when $k>2$, and to $G(B(2))\cdot 3 \zeta_3^a$ when $k=2$. 
For the stratum $A_1$ we have by \eqref{eqn: split of pairing at A1 stratum} 
\begin{eqnarray*}
& & 
\sum_{[y]\in {\OA}\backslash A_1}e(-q(y)) \langle [y], [y] \rangle_A \\ 
& = &  
\left( \sum_{[xe]\in pC/-1}e(-q(xe)) \langle [xe], [xe] \rangle_C \right) \cdot 
\left( \sum_{[w]\in {\rm O}(B)\backslash B}e(-q(w)) \langle [w], [w] \rangle_B \right). 
\end{eqnarray*}
The second term is $G'(B, {\rm O}(B))$. 
When $(p, k)\ne (3, 2)$, the first term is equal to $G'(C, {\rm O}(C))$ by Lemma \ref{redundant Gauss sum =0}.  
Let $(p, k)=(3, 2)$. 
Then the first term is equals to $3$. 
If $B$ contains $U$, we may use Proposition \ref{p-elementary 2nd kind} to obtain 
\begin{eqnarray*}
{\gaoa} 
&=& 3 \sqrt{-1}^{-m} \left( \frac{d(B)}{3} \right) \sqrt{|B|} \cdot (1+\zeta_3^a) \\ 
&=& G'(C, {\rm O}(C)) \cdot G'(B, {\rm O}(B)). 
\end{eqnarray*}
If $B$ is anisotropic, the decomposition $A=C\oplus B$ is canonical and we can use Lemma \ref{product decomposition Gauss sum}. 

Next let $p=2$. 
When $k\geq4$, the above calculation is still valid. 
When $k=3$, we have as in the proof of Proposition \ref{prop:(p,k)=(2,3)}
\begin{equation*}
\sum_{[x]\in {\OA}\backslash A_1} e(-q(x)) \langle [x], [x] \rangle = G'(C, {\rm O}(C)) \cdot G'(B, {\rm O}(B)), 
\end{equation*}
\begin{equation*}
\sum_{[x]\in {\OA}\backslash A_{0+}} e(-q(x)) \langle [x], [x] \rangle = 
(\zeta_{16}^{a}+\zeta_{16}^{5a}+\zeta_{16}^{9a}+\zeta_{16}^{13a})\sum_{v\in B_+}(-1)^{(v, v)} = 0,  
\end{equation*}
and similarly 
$\sum_{{\OA}\backslash A_{0-}} e(-q(x)) \langle [x], [x] \rangle = 0$. 
When $k=2$, we have $G'(C, {\rm O}(C))=0$ by Proposition \ref{prop: cyclic 2nd kind}. 
We shall show that ${\gaoa}=0$. 
As in the proof of Proposition \ref{prop:(p,k)=(2,2)}, 
${\gaoa}$ can be reduced  to 
\begin{equation*}
{\GAOA} = \sum_{\stackrel{[x]\in{\OA}\backslash A}{2x=0}} e(-q(x)) \langle [x], [x] \rangle_A, 
\end{equation*}
and we may assume that $B'\ne0$. 
When $x_B$ is nonzero and isotropic, \eqref{eqn: prod (p,k)=(2,2) 2} and \eqref{eqn: prod (p,k)=(2,2) 3} show that 
the contribution from orbits other than $0$, $x_B$, $2e$ and $x_B+2e$ cancels to $0$. 
The contribution from those elements also amounts to $0$, so we have ${\gaoa}=0$. 
The case of other $x_B$ is similar. 
\end{proof}

\begin{remark}
We also checked that ${\gaoa}=0$ for 
$A=A_{2^k,a}\oplus A_{4,b}$ and $A=A_{2^k,a}\oplus A_{4,b}\oplus A_{2,c}$ where $k\geq2$. 
Since $G'(A_{4,b}, {\rm O}(A_{4,b}))=0$ by Proposition \ref{prop: cyclic 2nd kind}, 
the product formula also holds in this case. 
\end{remark}

\section{Modular forms for the Weil representation}\label{sec:dim formula}

In this section we explain that the equivariant Gauss sums ${\GAOA}$ and ${\gaoa}$ appear 
in the dimension formula for certain vector-valued modular forms. 
This was our original motivation to study them. 

Let $A=(A, q)$ be a finite quadratic form. 
We write $\sigma(A)\in{\Z}/8$ for its signature. 
We recall the Weil representation associated to $A$ (see \cite{Bo}). 
Let ${\Mp}$ be the metaplectic double cover of ${\SL}$. 
Its elements are pairs $(M, \phi(\tau))$ where 
$M=\begin{pmatrix}a & b \\ c & d \end{pmatrix}\in {\SL}$ 
and $\phi(\tau)$ is a holomorphic function on the upper half plane such that $\phi(\tau)^2=c\tau+d$. 
The group ${\Mp}$ is generated by the elements  
\begin{equation*}
T = \left( \begin{pmatrix}1&1\\ 0&1\end{pmatrix}, 1 \right), \quad  
S = \left( \begin{pmatrix}0&-1\\ 1&0\end{pmatrix}, \sqrt{\tau} \right). 
\end{equation*}
Its center is cyclic of order $4$ generated by 
\begin{equation*}
Z = S^2 = \left( \begin{pmatrix}-1&0\\ 0&-1\end{pmatrix}, \sqrt{-1} \right). 
\end{equation*}
Let ${\C}A$ be the group algebra over $A$. 
We denote by $\mathbf{e}_x\in {\C}A$ the standard basis vector corresponding to $x\in A$. 
The Weil representation $\rho_A$ is the unitary representation of ${\Mp}$ on ${\C}A$ defined by 
\begin{eqnarray*}
\rho_A(T)(\mathbf{e}_x) & = & e(q(x)){\e}_x, \\  
\rho_A(S)(\mathbf{e}_x) & = & \frac{\zeta_{8}^{-\sigma(A)}}{\sqrt{|A|}} \sum_{y\in A}e(-(x, y)){\e}_y. 
\end{eqnarray*}
The orthogonal group ${\OA}={\rm O}(A, q)$ acts on ${\C}A$ by permuting the basis vectors ${\e}_x$, 
namely $\gamma({\e}_x)={\e}_{\gamma x}$ for $\gamma\in{\OA}$. 
This action commutes with $\rho_A$:  
\begin{equation*}\label{eqn:Weil rep and O(A)action}
\rho_A(T)\circ\gamma = \gamma\circ\rho_A(T), \qquad \rho_A(S)\circ\gamma = \gamma\circ\rho_A(S). 
\end{equation*}

Let  $l\in\frac{1}{2}{\Z}$. 
A ${\C}A$-valued holomorphic function $f(\tau)$ on the upper half plane is called a modular form of type $\rho_A$ and weight $l$ 
if it satisfies 
\begin{equation*}
f(M\tau) = \rho_A(M, \phi) \phi(\tau)^{2l} f(\tau), \quad (M, \phi)\in{\Mp}, 
\end{equation*}
and is holomorphic at the cusp. 
We write $M_l(\rho_A)$ for the space of such modular forms. 
The group ${\OA}$ acts on $M_l(\rho_A)$ naturally through its action on ${\C}A$. 
Our purpose is to compute the dimension of the invariant part $M_l(\rho_A)^{{\OA}}$. 
This is related to constructing modular forms for the full orthogonal group of an even lattice by means of lifting \cite{Bo}. 
Let $({\C}A)^{{\OA}}$ be the ${\OA}$-invariant subspace of ${\C}A$. 
As its basis we can take 
\begin{equation*}
v_{[x]} = \sum_{y\in{\OA}x}{\e}_y, \qquad [x]\in{\OA}\backslash A. 
\end{equation*}
Since the $\rho_A$-action commutes with the ${\OA}$-action, it preserves $({\C}A)^{{\OA}}$. 
We write $\rho_A^{inv}$ for the restriction of $\rho_A$ on $({\C}A)^{{\OA}}$. 
Then 
\begin{equation*}
M_l(\rho_A)^{{\OA}} = M_l(\rho_A^{inv}), 
\end{equation*}
so the problem is to compute the dimension of the space of $\rho_A^{inv}$-valued modular forms. 
We write 
\begin{equation*}
d_A^{inv} = {\dim}({\C}A)^{{\OA}} = |{\OA}\backslash A|. 
\end{equation*}
If $\lambda\in{\Q}/{\Z}$, we denote by $\alpha(\lambda)$ its representative in $[0, 1)$. 
We set  
\begin{equation*}
\alpha(A) = \sum_{[x]\in {\OA}\backslash A} \alpha(q(x)). 
\end{equation*}

\begin{proposition}
Let $l\equiv\sigma(A)/2$ mod $2{\Z}$ with $l\geq 2$. 
Then 
\begin{eqnarray*}\label{eqn:modular form dim}
{\dim}(M_l(\rho_A^{inv})) 
& = & 
 \frac{d_A^{inv}(l+5)}{12} - \alpha(A) + 
 (-1)^{(2l-\sigma(A))/4} \frac{{\GAOA}}{4\sqrt{|A|}}     \\ 
&  & +  \frac{2}{3\sqrt{3}\sqrt{|A|}} {\rm Re} \{ \zeta_{24}^{4l+2-3\sigma(A)} \overline{{\gaoa}} \}.  
\end{eqnarray*}
When $l\not\equiv\sigma(A)/2$ mod $2{\Z}$, we have $M_l(\rho_A^{inv})=\{ 0 \}$. 
\end{proposition}

\begin{proof}
Since $\rho_A(Z)$ maps ${\e}_x$ to $\sqrt{-1}^{-\sigma(A)}{\e}_{-x}$, we have 
\begin{equation*}
\rho_A(Z)({\e}_x + {\e}_{-x}) = \sqrt{-1}^{-\sigma(A)} ({\e}_x + {\e}_{-x}). 
\end{equation*}
The invariance under $Z$ means $\sqrt{-1}^{2l-\sigma(A)}=1$, 
so we must have $2l-\sigma(A)\in 4{\Z}$. 
In general, for a representation $\rho$ of ${\Mp}$ such that $\rho(Z)$ acts by the scalar multiplication by $\sqrt{-1}^{2l}$, 
a general dimension formula for $\rho$-valued modular forms of weight $l$ is given by Skoruppa (see, e.g., \cite{E-S} p.~129). 
In the present case, it takes the form 
\begin{eqnarray*}\label{eqn:cusp form dim}
{\dim}(M_l(\rho_A^{inv})) 
& = & 
\frac{d_A^{inv}(l+5)}{12} - \alpha(A)  +  \frac{1}{4} {\rm Re}( \zeta_{8}^{2l} {\tr}(\rho_A^{inv}(S)) ) \\ 
&  & + \frac{2}{3\sqrt{3}}   {\rm Re}( \zeta_{12}^{2l+1} {\tr}(\rho_A^{inv}(ST))).  
\end{eqnarray*}
We shall show that 
\begin{eqnarray*}
{\tr}(\rho_A^{inv}(S))    & = & \zeta_{8}^{-\sigma(A)} \sqrt{|A|}^{-1} \cdot {\GAOA},  \\  
 {\tr}(\rho_A^{inv}(ST)) & = & \zeta_{8}^{-\sigma(A)} \sqrt{|A|}^{-1} \cdot \overline{{\gaoa}}. 
\end{eqnarray*}
We use the natural basis $v_{[x]}$ of $({\C}A)^{{\OA}}$ to compute the traces. 
Then  
\begin{eqnarray*}
\zeta_{8}^{\sigma(A)} \sqrt{|A|} \cdot \rho_A(S)(v_{[x]}) 
& = & \sum_{x'\in{\OA}x} \sum_{y\in A} e(-(x', y))\mathbf{e}_{y} \\ 
& = & \sum_{[y]\in {\OA}\backslash A} \sum_{y'\in{\OA}y} \sum_{x'\in{\OA}x} e(-(x', y'))\mathbf{e}_{y'} \\ 
& = & \sum_{[y]\in {\OA}\backslash A} \sum_{y'\in{\OA}y} \langle [y'], [x] \rangle_A \mathbf{e}_{y'} \\ 
& = & \sum_{[y]\in {\OA}\backslash A} \langle [y], [x] \rangle_A v_{[y]}. 
\end{eqnarray*}
This gives the trace formula for $\rho_A^{inv}(S)$. 
For $ST$ we have 
\begin{equation*}
\rho_A(ST)(\mathbf{e}_{x}) = 
\zeta_{8}^{-\sigma(A)} \sqrt{|A|}^{-1} \cdot \sum_{y\in A} e(q(x))e(-(x, y))\mathbf{e}_{y} 
\end{equation*}
and hence 
\begin{eqnarray*}
\zeta_{8}^{\sigma(A)} \sqrt{|A|} \cdot \rho_A(ST)(v_{[x]}) 
& = & \sum_{[y]\in {\OA}\backslash A} \sum_{y'\in{\OA}y} \sum_{x'\in{\OA}x} e(q(x)) e(-(x', y'))\mathbf{e}_{y'} \\ 
& = & \sum_{[y]\in {\OA}\backslash A} \sum_{y'\in{\OA}y} e(q(x)) \langle [y'], [x] \rangle_A \mathbf{e}_{y'} \\ 
& = & \sum_{[y]\in {\OA}\backslash A} e(q(x)) \langle [y], [x] \rangle_A v_{[y]}. 
\end{eqnarray*} 
This gives the trace formula for $\rho_A^{inv}(ST)$. 
\end{proof}


\end{document}